\documentclass[12pt, psamsfonts]{amsart}
\usepackage{amssymb,amsfonts,amsthm,amsmath,epsfig,hhline}
\usepackage{verbatim, enumerate,bookmark}    

\usepackage[arc,all,graph,frame]{xy}

\theoremstyle{plain}
\newtheorem{theorem}{Theorem}

\newtheorem{corollary}{Corollary}

\newtheorem{lemma}{Lemma}

\theoremstyle{definition}
\newtheorem{definition}{Definition}
\newtheorem{example}{Example}
\newtheorem{remark}{Remark}

\newcommand{\N}{\mathbb{N}}

\def\d{\underbar{\em d}}

\title[Bipartite graph vertex sequences]{A sharp refinement of a result of {A}lon, {B}en-{S}himon and {K}rivelevich on bipartite graph vertex sequences}

\author{Grant Cairns}
\author{Stacey Mendan}
\author{Yuri Nikolayevsky}

\address{Dept of Mathematics and Statistics, La Trobe University, Melbourne, Australia 3086}
\email{G.Cairns@latrobe.edu.au}
\email{spmendan@students.latrobe.edu.au}
\email{Y.Nikolayevsky@latrobe.edu.au}

\begin{document}

\maketitle

\begin{abstract} We give a sharp refinement of a result of Alon, Ben-Shimon and Krivelevich. This gives a sufficient condition for a finite sequence of positive integers to be the vertex degree list of both parts of a bipartite graph. The condition depends only on the length of the sequence and its largest and smallest elements.

\end{abstract}

\section{Introduction}

Recall that a finite sequence $\d=(d_1,\dots,d_n)$ of positive integers is
\emph{graphic} if there is a simple graph with $n$ vertices having $\d$ as its list of  vertex degrees.
A pair  $(\d_1,\d_2)$ of sequences (possibly of different length) is \emph{bipartite graphic} if there is a simple, bipartite graph whose parts have $\d_1,\d_2$ as their respective lists of  vertex degrees.
We say that a sequence $\d$ is \emph{bipartite graphic} if the pair $(\d,\d)$ is bipartite graphic; that is, if there is a simple, bipartite graph whose two parts each have $\d$ as their list of  vertex degrees. The classic Erd\H{o}s--Gallai Theorem gives a necessary and sufficient condition for a sequence to be graphic. Similarly, the Gale--Ryser Theorem~\cite{Gale,Ryser} gives a necessary and sufficient condition for a pair of sequences to be bipartite graphic. In particular, the Gale--Ryser Theorem gives a necessary and sufficient condition for a single sequence to be bipartite graphic.

In \cite[Theorem~6]{ZZ}, Zverovich and Zverovich gave a sufficient condition, for a sequence to be graphic, depending only on the length of the sequence and its largest and smallest elements. A sharp refinement of this result is given in \cite{CMN}. In \cite[Corollary~2.2]{ABK}, Alon, Ben-Shimon and Krivelevich gave a result for bipartite graphic sequences, which is directly analogous to the theorem of Zverovich--Zverovich. The purpose of the present paper is to give a sharp refinement of the Alon--Ben-Shimon--Krivelevich result.

Here is the Alon--Ben-Shimon--Krivelevich result:

\begin{theorem}[{\cite[Corollary~2.2]{ABK}}]\label{T:ABK}
Suppose that $\emph{\d}$ is a finite sequence of positive integers having length $n$, maximum element $a$ and minimum element $b$. If for a real number $x\geq1$, we have
\begin{equation}\label{E:ABK}
a\leq\min\left\{xb,\frac{4xn}{(x+1)^2}\right\},
\end{equation}
then $\emph{\d}$ is bipartite graphic. \end{theorem}

As we will explain at the end of this introduction, Theorem \ref{T:ABK} can be rephrased in the following equivalent form:
\begin{theorem}\label{T:ABKsimp}
Suppose that $\emph{\d}$ is a finite sequence of positive integers having length $n$, maximum element $a$ and minimum element $b$.
Then $\emph{\d}$ is bipartite graphic if
\begin{equation}\label{E:ABKsimp}
nb\geq\frac{(a+b)^2}{4}.
\end{equation}
\end{theorem}

The main aim of this paper is to prove the following result.

\begin{theorem}\label{T:ABKsharp}
Suppose that $\emph{\d}$ is a finite sequence of positive integers having length $n$, maximum element $a$ and minimum element $b$.
 Then $\emph{\d}$ is bipartite graphic if
\begin{equation}\label{E:ABKgeneral}
nb\geq  \begin{cases}
\frac{(a+b)^2}4 &:\ \text{if } a\equiv b\pmod 2,\\ \\
\left\lfloor\frac{(a+b)^2}4\right\rfloor &:\ \text{otherwise},
\end{cases}
\end{equation}
where $\lfloor.\rfloor$ denotes the integer part. Moreover, for any triple $(a,b,n)$ of positive integers with $b<a\leq n$ that fails \eqref{E:ABKsimp}, there is a non-bipartite-graphic sequence of length $n$ with maximal element $a$ and minimal element $b$.
\end{theorem}

Let us contrast the above result with the sharp result for graphic sequences given in \cite{CMN}. We will require this result later in Section \ref{S:proofii}.

\begin{theorem}[\cite{CMN}]\label{T:zz2}
Suppose that $\emph{\d}$ is a finite sequence of positive integers having length $n$, maximum element $a$ and minimum element $b$.
 Then  $\emph{\d}$  is graphic if
\begin{equation}\label{E:zz2}
nb\geq  \begin{cases}
\left\lfloor\dfrac{(a  +  b +1)^2}{4    }\right\rfloor-1&: \ \text{if } b\ \text{is odd, or }a +b \equiv 1\pmod 4,\\ \\
\left\lfloor\dfrac{(a  +  b +1)^2}{4    }\right\rfloor&: \ \text{otherwise}.
\end{cases}
\end{equation}
Moreover, for any triple $(a,b,n)$ of positive integers with $b< a<n$ that fails \eqref{E:zz2}, there is a non-graphic sequence of length $n$ having even sum with maximal element $a$ and  minimal element $b$.
\end{theorem}

We give two proofs of Theorem \ref{T:ABKsharp}. The first proof is in the spirit of the original paper of Zverovich and Zverovich, and uses the notion of \emph{strong indices}. The preparatory results for this proof, notably Theorem \ref{T:ZZ7strongindices} and Lemma \ref{L:ABKsharpproof}, may be of independent interest. Our second proof is much shorter, and uses the sharp version of Zverovich--Zverovich from \cite{CMN} and recent results relating bipartite graphic sequences to the degree sequences of graphs having at most one loop at each vertex \cite{CM2}.

The paper is organised as follows.  Section~\ref{S:2e} gives a necessary and sufficient condition for a sequence of the form $(a^s,b^{n-s})$ to be bipartite graphic. Here and throughout the paper, the superscripts indicate the number of repetitions of the element. So, for example, the sequence $(5,5,5,4,4)$ is denoted $(5^3,4^2)$. In \ Section~\ref{S:2e}  we also prove Theorem \ref{T:ABKsharp}  for sequences of the form $(a^s,b^{n-s})$, and we give examples showing that Theorem \ref{T:ABKsharp} is sharp.
Section~\ref{S:BGcsi} presents results about bipartite graphic sequences, which are used in the first proof of Theorem~\ref{T:ABKsharp} found in Section~\ref{S:proofi}. Section~\ref{S:proofii} presents the second proof of Theorem~\ref{T:ABKsharp}.

To complete this introduction, let us establish the equivalence of Theorems \ref{T:ABK} and \ref{T:ABKsimp}. If $nb\geq\frac{(a+b)^2}{4}$, then setting $x=\frac{a}{b}$, we have that \eqref{E:ABK} holds. Thus  Theorem \ref{T:ABKsimp} follows from Theorem \ref{T:ABK}. Conversely, fix $a,b,n$ and note that the hypothesis of Theorem \ref{T:ABK} is that $a\leq x b$ and $a\leq \frac{4xn}{(x+1)^2}$. Observe that $\frac{4xn}{(x+1)^2}$ is a monotonic decreasing function of $x$ for $x\geq 1$. So if $a\leq \frac{4xn}{(x+1)^2}$ holds for some $x\geq \frac{a}b$, then $a\leq \frac{4xn}{(x+1)^2}$ holds for $x= \frac{a}b$, in which case \eqref{E:ABKsimp} holds. Hence Theorem \ref{T:ABK} follows from Theorem \ref{T:ABKsimp}.

\section{Two-element sequences}\label{S:2e}

We consider two-element sequences; that is, sequences of the form $(a^s,b^{n-s})$.

\begin{theorem}\label{T:ABKab}
Let $a,b,n,s \in \N$ with $b<a\leq n$ and $s\leq n$. Then the sequence $(a^s,b^{n-s})$ is bipartite graphic if and only if $s^2-(a+b)s+nb\geq0$.
\end{theorem}

\begin{proof}
We will employ \cite[Theorem~8]{ZZ}, from which we have in particular: a two-element sequence $\d=(a^s,b^{n-s})$  is bipartite graphic if and only if
\begin{equation}\label{ineqs}
\sum_{i=1}^s(a+in_{s-i}) \leq sn\qquad \text{and}\qquad \sum_{i=1}^s(a+in_{n-i})+\sum_{i=s+1}^n(b+in_{n-i})\leq n^2,
\end{equation}
 where $n_j$ is the number of elements of $\d$ equal to $j$; that is,
\[
n_j=\begin{cases}
s&:\ \text{if}\ j=a\\
n-s&:\ \text{if}\ j=b\\
0&:\ \text{otherwise}.
\end{cases}
\]
Notice that the second inequality in \eqref{ineqs} is always satisfied.
Indeed, \begin{align*}
\sum_{i=1}^s(a+in_{n-i})+\sum_{i=s+1}^n(b+in_{n-i})&=as+(n-s)b+\sum_{j=0}^{n-1}(n-j)n_j\\
&= s(a-b)+nb+(n-a)s+(n-b)(n-s)=n^2.
\end{align*}
So, rewriting the first inequality in \eqref{ineqs}, we have that $\d=(a^s,b^{n-s})$  is bipartite graphic if and only if
\begin{equation}\label{ineqf}
\sum_{j=0}^{s-1}(s-j)n_j   \leq s(n-a).
\end{equation}
If $b<s\leq a$, then
$\sum_{j=0}^{s-1}(s-j)n_j=(s-b)(n-s)$ and hence
 \[\sum_{j=0}^{s-1}(s-j)n_j   \leq s(n-a) \iff s^2-(a+b)s+nb\geq0,\]
as required. It remains to consider the cases $s\leq b$ and $a<s$. If $s\leq b$, then \[\sum_{j=0}^{s-1}(s-j)n_j=0\leq s(n-a).\]\ If $a<s$, then
\[
\sum_{j=0}^{s-1}(s-j)n_j=(s-a)s+(s-b)(n-s)=s(n-a)-b(n-s)\leq s(n-a).
\]
The inequality $s^2-(a+b)s+nb\geq0$ holds in both these cases. Indeed, the minimum of the function $f(s)=s^2-(a+b)s+nb$ occurs at $s=\frac{a+b}2$ so $f(s)$ is decreasing for $s\leq b$, and increasing for $a<s$, and $f(a)=f(b)=(n-a)b\geq0$.
\end{proof}

\begin{example}
First assume $a\equiv b\pmod2$ and $4nb<(a+b)^2$. Then the sequence \[(a^{\frac{a+b}2},b^{\frac{2n-a-b}2})\] is not bipartite graphic by Theorem~\ref{T:ABKab}. Now assume $a\not\equiv b\pmod2$ and $4nb<(a+b)^2-1$. Then \[(a^{\frac{a+b+1}2},b^{\frac{2n-a-b-1}2})\] is not bipartite graphic, again by Theorem~\ref{T:ABKab}. These examples show that the bound given in Theorem \ref{T:ABKsharp} is sharp.
\end{example}

\begin{remark}\label{R:equiv}
Note that for two-element sequences, we can deduce Theorem \ref{T:ABKsharp} from Theorem \ref{T:ABKab}. Indeed, suppose that
$\d=(a^s,b^{n-s})$ and that
\[
nb\geq  \begin{cases}
\frac{(a+b)^2}4 &:\ \text{if } a\equiv b\pmod 2,\\ \\
\left\lfloor\frac{(a+b)^2}4\right\rfloor &:\ \text{otherwise}.
\end{cases}
\]
As we observed in the proof of Theorem \ref{T:ABKab},
the minimum of the function $f(s)=s^2-(a+b)s+nb$ occurs at $\frac{a+b}2$.
If $a+b$ is even, then
\[
f(s)\geq f\left(\frac{a+b}2\right)=nb-\frac{(a+b)^2}4\geq 0,
\]
and so $\d$ is bipartite graphic by Theorem \ref{T:ABKab}.
So we may suppose that $a+b$ is odd. Then as $s$ is an integer,
\[
f(s)\geq f\left(\frac{a+b-1}2\right)=nb-\frac{(a+b)^2-1}4=nb- \left\lfloor\frac{(a+b)^2}4\right\rfloor \geq 0.
\]
Hence $\d$ is bipartite graphic by Theorem \ref{T:ABKab}.
\end{remark}

\section{Strong indices}\label{S:BGcsi}

In this section, $\d=(d_1,\dots,d_n)$ is a decreasing sequence of positive integers and for each integer $j$,  the number of elements in $\d$ equal to $j$ is denoted $n_j$.
As a particular case of \cite[Theorem~7]{ZZ}, one has the following.

\begin{theorem}[{\cite{ZZ}}]\label{T:ZZ7}
The sequence  $\emph{\d}$ is bipartite graphic if and only if $\sum_{i=1}^k(d_i+in_{k-i})\leq kn$, for all  indices $k$.
\end{theorem}

Recall the following standard definition.
\begin{definition}
In the sequence $\d$, an index is said to be \emph{strong} if $d_k\geq k$.
\end{definition}

The following result improves Theorem~\ref{T:ZZ7}.

\begin{theorem}\label{T:ZZ7strongindices}
The sequence  $\emph{\d}$  is bipartite graphic if and only if $\sum_{i=1}^k(d_i+in_{k-i})\leq kn$, for all strong indices $k$.
\end{theorem}

\begin{proof}
Necessity follows from Theorem~7 in~\cite{ZZ}. To prove sufficiency, define
\[
    F_k=kn-\sum_{i=1}^k(d_i+in_{k-i})=kn-\sum_{i=1}^kd_i-\sum_{i=0}^k(k-i)n_i.
\]
Suppose that $F_{k}\geq0$ for all strong indices $k$. We will show that $F_{k}\geq0$ for all indices $k$. To do this, we show that the minimum value of
$F_{k}$, for $k=1,2,\dots,n$, is nonnegative, and to do this we look at the smallest $k$ for which $F_{k}$ assumes the minimum value. Thus it suffices  to show that $F_1$ and $F_n$ are nonnegative and $F_k\geq0$ for all $k=2,\dots,n-1$ such that $F_{k-1}>F_k$ and $F_{k+1}\geq F_k$.
We will make use of the following lemma. Define the function $f:\N \cup \{0\} \to \N \cup \{0\}$ as follows: $f(k)=\max\{p \, : \, d_p \ge {k+1}\}$, with the convention that $\max \varnothing = 0$.
\begin{lemma}\label{L:ZZsiproof}
For the sequence $\emph{\d}$, suppose that $n \ge d_1$. For a given $k=0,1,\dots, n$, denote $p=f(k)$. Then, in the above notation,
\begin{enumerate}[{\rm (a)}]
  \item \label{it:pfkonestrong}
  if $k,p>0$, then at least one of them is a strong index,

  \item \label{it:pfksumns}
  $\sum_{s=k+1}^nn_s=p$ and $\sum_{s=0}^nn_s=n$,

  \item \label{it:pfksumsns}
  $\sum_{s=k+1}^nsn_s=\sum_{i=1}^pd_i$ and $\sum_{s=0}^nsn_s=\sum_{i=1}^{n}d_i$,

  \item \label{it:pfksymm}
  $F_k=\sum_{i=1}^{n}d_i -\sum_{i=1}^{k}d_i -\sum_{i=1}^{p}d_i +kp$. In particular, if  $f(p)=k$, then $F_k=F_p$.
\end{enumerate}
\end{lemma}
\begin{proof}
(a) Suppose $k$ is not a strong index, so that $k > d_k$. As $p =f(k)$ is assumed to be positive we have $p \in \{1, \dots, n\}$ and moreover, $d_p \ge k+1 > d_k$. So, as $\d$ is decreasing, $p < k$. Thus  $d_p \ge k+1 > p$ and so $p$ is a strong index, as required.

(b) The left-hand side of the first equality equals $\#\{s\,: \,d_s \ge k+1\}=p$ by definition. The second equality is obvious.

(c) For an arbitrary $s \ge 0$ we have $sn_s=\sum_{i: d_i=s} d_i$. It follows that $\sum_{s=k+1}^nsn_s=\sum_{s=k+1}^n \sum_{i: d_i=s} d_i=\sum_{i: d_i \ge k+1} d_i =\sum_{i=1}^pd_i$. This proves the first equality; the second equality is obvious.


(d) We have by \eqref{it:pfksumns} and \eqref{it:pfksumsns}:
\begin{align*}
    F_k &=kn-\sum\nolimits_{i=1}^k d_i - k \sum\nolimits_{i=0}^k n_i + \sum\nolimits_{i=0}^k i n_i \\
    &= k \Big(n-  \sum\nolimits_{i=0}^k n_i\Big) -\sum\nolimits_{i=1}^k d_i + \sum\nolimits_{i=0}^n i n_i - \sum\nolimits_{i=k+1}^nin_i \\
    &= kp -\sum\nolimits_{i=1}^{k}d_i + \sum\nolimits_{i=1}^{n}d_i-\sum\nolimits_{i=1}^{p}d_i ,
\end{align*}
as required. If not only $f(k)=p$, but also $f(p)=k$, then $F_k=F_p$, as the latter expression for $F_k$ is symmetric with respect to $k$ and $p$.
\end{proof}
Continuing with the proof of the theorem, by Lemma~\ref{L:ZZsiproof}\eqref{it:pfksumns},
\begin{equation}\label{eq:FminusF}
    F_{k+1}-F_k=n-d_{k+1} - \sum\nolimits_{i=0}^k n_i=\sum\nolimits_{i=k+1}^n n_i-d_{k+1}= f(k) - d_{k+1}.
\end{equation}
Moreover, $F_n=n^2-\sum_{i=1}^n d_i- n\sum_{i=0}^n n_i+ \sum_{i=0}^n in_i=0$ by Lemma~\ref{L:ZZsiproof}(\ref{it:pfksumns}, \ref{it:pfksumsns}) and $F_1\geq0$ by assumption, as $d_1\geq1$. By \eqref{eq:FminusF} and Lemma~\ref{L:ZZsiproof}(\ref{it:pfksumns}), the inequalities $F_{k-1}>F_k$ and $F_{k+1}\geq F_k$ give
\begin{align*}
    F_{k+1}-F_k&=f(k) - d_{k+1} \ge 0, \\
    F_k-F_{k-1}&=f(k-1) - d_{k} = f(k)+ n_k - d_k < 0.
\end{align*}
That is,
\begin{equation}\label{eq:dkdk+1}
 d_{k+1} \le f(k) < d_k-n_k.
\end{equation}
Let $k$ be a non-strong index for which~\eqref{eq:dkdk+1} holds. Denote $p=f(k)$. If $p>0$, then $p$ is a strong index by Lemma~\ref{L:ZZsiproof}\eqref{it:pfkonestrong}, hence $F_p \ge 0$ by assumption. Moreover, by \eqref{eq:dkdk+1} we have $d_{k+1} \le p$ and $d_k > p+n_k$ so $d_k \ge p+1$ and $d_{k+1} < p+1$. It follows that $k=\max\{s \, : \, d_s \ge {p+1}\}$, so $f(p)=k$ by definition. Then, by  Lemma~\ref{L:ZZsiproof}\eqref{it:pfksymm}, we have $F_k=F_p \ge 0$. So we may assume that  $p=0$. Then $d_{k+1}=0$, by \eqref{eq:dkdk+1}, and hence $d_j = 0$ for all $j > k$. Furthermore,  as $f(k)=p=0$, we have $\{s\,:\,d_s\geq k+1\}=\varnothing$,  and so $n_i = 0$ for all $i>k$. So by \eqref{eq:FminusF}, for every $j > k$ we have $F_j-F_{j-1}=\sum_{i=j}^n n_i-d_j=0$. Thus $F_k=F_n$. As $F_n=0$ from the above, we get  $F_k=0$, as required.
\end{proof}

In the next section, we will also need the following lemma, which is a variation of  \cite[Lemma~1]{CMN}.

\begin{lemma}\label{L:ABKsharpproof}
Suppose that $\emph{\d}$ has maximum element $a=d_1\leq n$ and minimum element $b=d_n$. For every strong index $k>b$, we have
\[\sum_{i=1}^k(d_i+in_{k-i})\leq n(k-b)+K (a+b)-K ^2,\] where $K$ is the largest strong index, $K=\max\{k:d_k\geq k\}$.
\end{lemma}

\begin{proof}
Let $k>b$ be a strong index.  We have $\sum_{i=1}^k d_i\leq ka$. Furthermore, since $n_j=0$ for $j<b$, we have
\[
\sum_{i=1}^kin_{k-i}=\sum_{j=0}^{k-1}(k-j)n_j\leq (k-b)\sum_{j=0}^{k-1}n_j.
\]
 The sum $\sum_{j=0}^{k-1}n_j$ counts the number of elements of $\d$ strictly less than $k$, hence $\sum_{j=0}^{k-1}n_j\leq n-K $ as $d_{K }\geq K \geq k$. Hence
\begin{equation}\label{eq:rk}
\sum_{i=1}^k(d_i+in_{k-i})\le ka+(k-b)(n-K ).
\end{equation}
As $a\geq d_{K }\geq K $, we have $a+1-K \geq 1$. Thus, using $k\leq K $, inequality \eqref{eq:rk} gives
\begin{align*}
\sum_{i=1}^k(d_i+in_{k-i}) \le ka+(k-b)(n-K )&=kn+k(a-K )+bK -bn \\
&\le kn+K (a-K )+bK -bn\\
&=n(k-b)+K (a+b)-K ^2,
\end{align*}
as required.
\end{proof}

\section{First Proof  of Theorem~\ref{T:ABKsharp}}\label{S:proofi}

Let $\d$ be a sequence satisfying hypothesis \eqref{E:ABKgeneral} of Theorem~\ref{T:ABKsharp}. If $a\equiv b\pmod 2$, then the result follows from Theorem~\ref{T:ABKsimp}. So we may assume that $a,b$ have different parity. Let $k$ be a strong index and suppose first that $k>b$. By Lemma~\ref{L:ABKsharpproof},
\begin{equation}\label{E:le}
\sum_{i=1}^k(d_i+in_{k-i})\leq n(k-b)+K (a+b)-K ^2,
\end{equation}
where $K$ denotes the largest strong index. As a quadratic in $K$, the maximal value of $n(k-b)+K (a+b)-K ^2$ is attained at $K =\frac{a+b\pm1}2$ and
\[
n(k-b)+\frac{(a+b\pm1)}2(a+b)-\left(\frac{a+b\pm1}2\right)^2=n(k-b)+\frac14(a+b)^2-\frac14.
\]
Hence, since $nb\geq\left\lfloor\frac{(a+b)^2}4\right\rfloor=\frac{(a+b)^2}4-\frac14 $, we have
 \[
 n(k-b)+K (a+b)-K ^2\leq n(k-b)+\frac14(a+b)^2-\frac14\leq kn.\]
 So by \eqref{E:le},  we have $\sum_{i=1}^k(d_i+in_{k-i})\leq kn$ and hence $\d$ is bipartite graphic by Theorem \ref{T:ZZ7strongindices}.
On the other hand, if $k\leq b$, then $\d$ contains no elements less than $k$ and hence
\begin{equation}\label{E:ka}
\sum_{i=1}^k(d_i+in_{k-i})=\sum_{i=1}^k d_i\leq ka.
\end{equation}
Note that $n\geq a$, since otherwise by \eqref{E:ABKgeneral}, we would have $ab> nb\geq \frac{(a+b)^2-1}4$, and hence $(a-b)^2<1$, giving $a=b$, which is impossible as $a,b$ have different parity.
So \eqref{E:ka} gives  $\sum_{i=1}^k(d_i+in_{k-i})\leq kn$ and once again,  $\d$ is bipartite graphic by Theorem \ref{T:ZZ7strongindices}.

\section{Second Proof  of Theorem~\ref{T:ABKsharp}}\label{S:proofii}

Suppose we have a decreasing sequence $\d=(a,\dots,b)$ of length $n$, and suppose it satisfies hypothesis \eqref{E:ABKgeneral} of Theorem \ref{T:ABKsharp}. By Remark \ref{R:equiv}, we may assume that $\d$ has at least 3 distinct elements.
Suppose that $n_a=s$; that is, $\d$ has precisely $s$ elements equal to $a$. Now consider the sequence $\d\,'$ obtained from $\d$ by reducing the first $s$ elements of $\d$ by 1. So $\d\,'$ has maximal element $a'=a-1$. Note that $\d$ has at least 3 distinct elements, hence the minimum element of $\d\,'$ is still $b$.
Suppose for the moment that $\d\,'$ has even sum. We will show that $\d\,'$ is graphic. From \eqref{E:ABKgeneral}, we have
\[
nb\geq  \begin{cases}
\frac{(a+b)^2}{4}&:\ \text{if } a\equiv b\pmod 2,\\ \\
\left\lfloor\frac{(a+b)^2}4\right\rfloor&:\ \text{otherwise},
\end{cases}
\]
We will show that
\begin{equation}\label{E:gs}
nb\geq  \begin{cases}
\left\lfloor\dfrac{(a'  +  b +1)^2}{4    }\right\rfloor-1&: \ \text{if } b\ \text{is odd, or }a' +b \equiv 1\pmod 4,\\ \\
\left\lfloor\dfrac{(a'  +  b +1)^2}{4    }\right\rfloor&: \ \text{otherwise},
\end{cases}
\end{equation}
from which we can conclude that $\d\,'$ is graphic by Theorem \ref{T:zz2}.
Consider two cases according to whether or not $a\equiv b \pmod 2$.
If  $a\equiv b \pmod 2$, then our hypothesis is $nb\geq \frac{(a+b)^2}{4}$, and hence
\[
nb\geq \frac{(a'+b+1)^2}{4} =\left\lfloor\dfrac{(a'  +  b +1)^2}{4    }\right\rfloor, \]
and so  \eqref{E:gs} holds. Similarly, if  $a\not\equiv b \pmod 2$, then our hypothesis is $nb\geq \left\lfloor\frac{(a+b)^2}4\right\rfloor$, and hence
\[
nb\geq \left\lfloor\frac{(a'+b+1)^2}{4}\right\rfloor , \]
and again \eqref{E:gs} holds.
Thus in either case, $\d\,'$ is graphic.

We now use a  result of \cite{CM2}. By a \emph{graph-with-loops} we mean a graph, without multiple edges, in which there is at most one loop at each vertex. For a graph-with-loops,
 the \emph{reduced degree} of a vertex is  taken to be the number of edges incident to the vertex, with \emph{loops counted once}. This differs from the usual definition of degree in which each loop contributes two to the degree.
By \cite[Corollary 1]{CM2},  a sequence  $\d$ of positive integers  is the sequence of  reduced degrees of the vertices of a graph-with-loops if and only if $\d$ is bipartite graphic. In our case, $\d\,'$ is graphic. Take a realization of $\d\,'$ as the degree sequence of some graph $G'$, and label the vertices of $G'$ in the same order as $\d\,'$.
Now add a loop to each of the first $s$ nodes of $G'$ and call the resulting graph-with-loops $G$. So the sequence of reduced degrees of $G$ is $\d$. Thus by \cite[Corollary 1]{CM2},  $\d$ is bipartite graphic.

It remains to deal with the case where $\d\,'$ has odd sum. Since $\d$ has at least 3 distinct elements,  we can modify the above construction as follows: we take the sequence $\d\,''$ obtained from $\d$ by reducing the first $(s+1)$ elements of $\d$ by 1. Then $\d\,''$ has even sum, maximum element $a-1$ and minimum element $b$, and we proceed exactly as above, only adding $s+1$ loops.
\providecommand{\bysame}{\leavevmode\hbox to3em{\hrulefill}\thinspace}
\providecommand{\MR}{\relax\ifhmode\unskip\space\fi MR }
\providecommand{\MRhref}[2]{%
  \href{http://www.ams.org/mathscinet-getitem?mr=#1}{#2}
}
\providecommand{\href}[2]{#2}

  \end{document}